\documentclass[10pt]{amsart}
\usepackage{amssymb, amsrefs,amsxtra,comment,xspace}
\usepackage[all]{xy}

\newtheorem{theorem}{Theorem}
\newtheorem{conjecture}{Conjecture}
\newtheorem{proposition}{Proposition}
\newtheorem{remark}{Remark}
\newtheorem{corollary}{Corollary}
\newtheorem{lemma}{Lemma}

\renewcommand{\AA}{{\mathbb{A}}}
\newcommand{\NN}{{\mathbb{N}}}
\newcommand{\ZZ}{{\mathbb{Z}}}
\newcommand{\QQ}{{\mathbb{Q}}}

\newcommand{\EE}{{\bar{\QQ}_\ell}}

\newcommand{\FF}{{\mathbb{F}}}
\newcommand{\ff}{{\Bbbk}}

\newcommand{\KKq}{{\bar{\Kq}}}
\newcommand{\tKq}{{\Kq^{\operatorname{tr}}}}

\newcommand{\RRq}{{\mathfrak{O}_{\KKq}}}

\newcommand{\kkq}{{\bar{\FF}_q}}

\newcommand{\K}{{\mathbb{K}}}
\newcommand{\KK}{{\bar{\mathbb{K}}}}

\newcommand{\Kq}{{\mathbb{K}}}
\newcommand{\Rq}{{\mathfrak{O}_{\Kq}}}

\newcommand{\kq}{{\FF_q}}

\newcommand{\sch}[1]{{#1}}
\newcommand{\Ksch}[1]{{{#1}}}
\newcommand{\KKsch}[1]{{{#1}_{\KKq}}}

\newcommand{\kksch}[1]{{{#1}_{\kkq}}}
\newcommand{\ffsch}[1]{{{#1}_{\ff}}}

\newcommand{\RSch}[2]{{\underline{#1}_{#2}}}
\newcommand{\RRSch}[2]{{(\underline{#1}_{#2})_\RRq}}
\newcommand{\KSch}[2]{{(\underline{#1}_{#2})_{\Kq}}}
\newcommand{\KKSch}[2]{{(\underline{#1}_{#2})_{\KKq}}}
\newcommand{\kSch}[2]{{(\underline{#1}_{#2})_{\kq}}}
\newcommand{\kkSch}[2]{{ (\underline{#1}_{#2})_{\kkq}}}
\newcommand{\rkSch}[2]{{(\underline{#1}_{#2})_{\kq}^{\operatorname{red}} }}
\newcommand{\rkkSch}[2]{{(\underline{#1}_{#2})_{\kkq}^{\operatorname{red}} }}

\newcommand{\Lq}{{\Kq'}}
\newcommand{\RLq}{{\mathfrak{O}_{\Kq'}}}
\renewcommand{\lq}{{\FF_{q'}}}

\newcommand{\rk}[3]{{\nu_\RSch{#1}{#3} \kSch{#2}{#3}}}
\newcommand{\rkk}[3]{{\nu_\RSch{#1}{#3} \kkSch{#2}{#3}}}

\newcommand{\LBsch}[1]{{\Ksch{#1}}}

\newcommand{\RLBSch}[2]{{ \underline{#1}_{#2} }}
\newcommand{\RRLBSch}[2]{{ {\underline{\bar #1}}_{#2} }}
\newcommand{\LBSch}[2]{{ (\underline{#1}_{#2})_{\Lq}}}
\newcommand{\LLBSch}[2]{{ (\underline{#1}_{#2})_{\KKq}}}
\newcommand{\lBSch}[2]{{ (\underline{#1}_{#2})_{\lq}}}
\newcommand{\llBSch}[2]{{ (\underline{#1}_{#2})_{\kkq} }}

\newcommand{\Lsch}[1]{{\Ksch{#1}_{\Lq}}}

\newcommand{\rlSch}[2]{{ (\underline{#1}_{#2})_{\lq}^{\operatorname{red}} }}

\newcommand{\fais}[1]{{ \mathcal{#1} }}

\newcommand{\rkPSch}[2]{{{#1}_{#2}}}
\newcommand{\rkkPSch}[2]{{ ({#1}_{#2})_\kkq }}

\def\cf{{{\it cf.\ }}}

\newcommand{\tq}{{\ \vert\ }}
\newcommand{\iso}{{\ \cong\ }}

\newcommand{\ceq}{{\, := \, }}

\newcommand{\res}{{\operatorname{res}}}
\newcommand{\ind}{{\operatorname{ind}}}
\newcommand{\Spec}[1]{{\operatorname{Spec}\left(#1\right)}}

\newcommand{\catM}{{\mathcal{M}}}
\newcommand{\proj}{{\operatorname{pr}}}

\newcommand{\id}{{\, \operatorname{id}}}

\newcommand{\cres}[2]{{\operatorname{cres}_\RRSch{#1}{#2} }}

\newcommand{\Sp}{{\operatorname{Sp}}}

\newcommand{\NC}[1]{{ {{\rm R}\Psi}_{#1}}}

\newcommand{\Gm}{{\mathbb{G}_{m}}}

\newcommand{\Gadd}{{\mathbb{G}_a}}
\newcommand{\Ext}{{\operatorname{Ext}}}
\newcommand{\IC}{{\operatorname{I\hskip-1pt C}^\bullet}}
\newcommand{\git}{/\!/}

\title[Toward a Mackey Formula for Compact Restriction]{Toward a Mackey Formula for Compact Restriction of Character Sheaves}

\date\today

\author{Pramod N. Achar}
\address{Department of Mathematics, Loiusiana State University}
\email{pramod@math.lsu.edu}

\author{Clifton L.R. Cunningham}
\address{Department of Mathematics, University of Calgary}
\email{cunning@math.ucalgary.ca}

\subjclass[2000]{}


\begin{document}

\begin{abstract}
We generalize \cite{cunningham-salmasian:sheaves}*{Theorem~3} to a Mackey-type formula for the compact restriction of a semisimple perverse sheaf produced by parabolic induction from a character sheaf, under certain conditions on the parahoric group scheme used to define compact restriction. This provides new tools for matching character sheaves with admissible representations. 
\end{abstract}

\maketitle


\section*{Introduction}

In this paper we prove a Mackey-type formula for the compact restriction functors introduced in \cite{cunningham-salmasian:sheaves}. The main result, Theorem~\ref{theorem: 3bis}, applies to any connected reductive linear algebraic group $\Ksch{G}$ over any non-Archimendean local field $\Kq$ that satisfies the following three hypotheses:
\begin{enumerate}
\item[(H.0)]
$\Ksch{G}$ is the generic fibre of a smooth, connected reductive group scheme over the ring of integers $\Rq$ of $\Kq$;
\item[(H.1)]
the characteristic of $\Kq$ is not $2$ (in particular, this condition is met if the characteristic of $\Kq$ is $0$);
\item[(H.2)]
for every parabolic subgroup $\KKsch{P}'\subseteq\Ksch{G}\times_\Spec{\Kq} \Spec{\KKq}$ there is a finite unramified extension $\Lq$ of $\Kq$ and a subgroup $\Ksch{P} \subseteq \Ksch{G}\times_\Spec{\Kq} \Spec{\Lq}$ such that $\Ksch{P}\times_\Spec{\Lq}\Spec{\KKq}$ is conjugate to $\KKsch{P}'$ by an element of $\Ksch{G}(\tKq)$. 
\end{enumerate}
Here, $\KKq$ is a separable algebraic closure of $\Kq$ and $\tKq$ is the maximal tamely ramified extension of $\Kq$ contained in $\KKq$.

As far as applications to representation theory are concerned, these are, arguably, mild hypotheses.
Hypothesis H.0 is equivalent to demanding that the Bruhat-Tits building of $\Ksch{G}(\Kq)$ admits a hyperspecial vertex (see \cite{tits:corvallis}). Every quasi-split reductive linear algebraic group over $\Kq$ that splits over an unramified extension of $\Kq$ satisfies this hypothesis (again, see \cite{tits:corvallis}).  Hypothesis H.1 is met whenever $\Kq$ is a finite extension of $\QQ_p$ and $p>2$. Hypothesis H.2 is satisfied if $\Ksch{G}$ is quasi-split over a maximal unramified extension of $\Kq$ and so, in particular, if $\Ksch{G}$ is quasi-split over $\Kq$. If $\Ksch{G}$ is specified, Hypothesis H.2 has the effect of imposing a lower bound on the residual characteristic of $\Kq$ that depends on $\Ksch{G}$. One large and interesting class of algebraic groups to which Theorem~\ref{theorem: 3bis} applies (because they satisfy Hypotheses H.0, H.1 and H.2) consists of unramified linear algebraic groups $\Ksch{G}$ over non-Archimedean local fields $\Kq$ of characteristic $0$ or greater than $3$. 

In order to state Theorem~\ref{theorem: 3bis}, we must recall a few facts concerning parahoric group schemes. In \cite{bruhat-tits:reductive-groups-1} and \cite{bruhat-tits:reductive-groups-2}, Fran\c{c}ois Bruhat and Jacques Tits showed that parahoric subgroups of $\Ksch{G}(\Kq)$ (where $\Ksch{G}$ is  a connected reductive linear algebraic group over $\Kq$) may be understood as subgroups arising from a class of smooth group schemes over $\Spec{\Rq}$ with generic fibre $\Ksch{G}$; these smooth integral models for $\Ksch{G}$ are habitually called \emph{parahoric group schemes}.
They further showed that parahoric group schemes are parametrized by facets in the Bruhat-Tits building for $\Ksch{G}(\Kq)$. Let $I(\Ksch{G},\Kq)$ denote the Bruhat-Tits building for $\Ksch{G}(\Kq)$ and for each $x\in I(\Ksch{G},\Kq)$, let $\RSch{G}{x}$ denote the parahoric group scheme attached to (the minimal facet  containing) $x$. Then $\RSch{G}{x}$ is a smooth group scheme over $\Spec{\Rq}$ and its generic fibre, $\KSch{G}{x}$, is $\Ksch{G}$. 
 The group $\RSch{G}{x}(\Rq)$ of $\Rq$-rational points on $\RSch{G}{x}$ is a parahoric subgroup of $\Ksch{G}(\Kq)$ and every parahoric subgroup of $\Ksch{G}(\Kq)$ arises in this manner. Although the special fibre of $\RSch{G}{x}$, denoted by $\kSch{G}{x}$ in this paper, is a connected linear algebraic group over the residue field $\kq$ of $\Kq$, it is generally not a reductive group scheme. \emph{Parahoric group schemes are generally not reductive group schemes}. In fact, $\RSch{G}{x}$ is reductive precisely when the parahoric subgroup $\RSch{G}{x}(\Rq)$ is hyperspecial; in this case, $x$ is a hyperspecial vertex in $I(\Ksch{G},\Kq)$. 
Even if $x$ is not hyperspecial, it is useful to consider the map (of group schemes over $\kq$) $\nu_\RSch{G}{x} : \kSch{G}{x} \to \rkSch{G}{x}$ to the maximal reductive quotient of $\kSch{G}{x}$. 

One more notion is required in order to state Theorem~\ref{theorem: 3bis}: the compact restriction functors introduced in \cite{cunningham-salmasian:sheaves}. These are designed with applications to characters of admissible representations in mind; here we recall their definition only. 
Each parahoric group scheme $\RSch{G}{x}$ determines a compact restriction functor \[\cres{G}{x} : D^b_c(\KKsch{G},\EE) \to D^b_c(\rkkSch{G}{x},\EE),\] introduced in \cite{cunningham-salmasian:sheaves}*{Definition~1} and defined by  
\[
\cres{G}{x} \ceq (\nu_{\RRSch{G}{x}})_!\ (\dim \nu_{\RSch{G}{x}}/2)\ \NC{\RRSch{G}{x}}.
\] 
Here  $\NC{\RRSch{G}{x}}: D^b_c(\KKsch{G},\EE) \to D^b_c(\kkSch{G}{x},\EE)$
is the nearby cycles functor (defined as in \cite{BBD}*{4.4.1}, for example) for the group scheme $\RRSch{G}{x} \ceq \RSch{G}{x} \times_\Spec{\Rq} \Spec{\RRq}$, where $\RRq$ is the ring of integers of a fixed separable algebraic closure $\KKq$ of $\Kq$, and $(\dim \nu_{\RSch{G}{x}}/2)$ indicates Tate twist by $\dim \nu_{\RSch{G}{x}}/2$. Notice that the compact restriction functor uses push-forward with compact supports of the morphism $\nu_\RRSch{G}{x} : \kkSch{G}{x} \to \rkkSch{G}{x}$ obtained from $\nu_\RSch{G}{x}$ by extending scalars from $\kq$ (the residue field of $\Kq$) to $\kkq$ (the residue field of $\KKq$). Hypothesis H.0 ensures that $\dim \nu_{\RSch{G}{x}}$ is even \cite{cunningham-salmasian:sheaves}*{Lemma~2}.
   
Now we may state Theorem~\ref{theorem: 3bis}, supposing Hypotheses H.1 and H.2 are met for $\Ksch{G}$ over $\Kq$: if $\Lq/\Kq$ is a finite unramified extension and if $\LBsch{P}$ is a parabolic subgroup of $\Ksch{G}\times_\Spec{\Kq} \Spec{\Lq}$ with reductive quotient $\Ksch{L}\times_\Spec{\Kq} \Spec{\Lq}$ (so $\Ksch{L}$ is a `twisted Levi subgroup' of $\Ksch{G}$), then for every $x \in I(\Ksch{G},\Kq)$ for which that star of $x\in I(G,\K)$ contains a hyperspecial vertex (in which case Hypothesis H.0 is also met) there is a finite set $\mathcal{S}\subset G(\Lq)$ such that 
\[
\cres{G}{x}\ind^\KKsch{G}_\KKsch{P}\ \fais{G}
\iso \mathop{\oplus}\limits_{g\in \mathcal{S}} \ind^\rkkSch{G}{x}_{\rkk{G}{\,^g\hskip-1pt P}{x'}}\ \,^g\hskip-1pt\left(\cres{L}{x'g} \fais{G}\right),
\]
for every character sheaf $\fais{G}$ of $\KKsch{L}\ceq \Ksch{L}\times_\Spec{\Kq} \Spec{\KKq}$.
The finite set $\mathcal{S}$, the parabolic subgroups $\rkk{G}{\,^g\hskip-1pt P}{x'}$ of $\rkkSch{G}{x}$, the integral models $\RSch{L}{x'g}$ appearing in $\cres{L}{x'g}$, and the meaning of $\,^g\hskip-1pt(\cres{L}{x'g} \fais{G})$, are all given in the proof of Theorem~\ref{theorem: 3bis}.

In \cite{cunningham-salmasian:sheaves} we showed that the compact restriction functors $\cres{G}{x}$ satisfy properties that go some way to showing that they are cohomological analogues of compact restriction functors for admissible representations. Theorem~\ref{theorem: 3bis} extends this analogy by providing a Mackey-type formula for $\cres{G}{x} \ind_\KKsch{P}^\KKsch{G} \fais{G}$ in certain cases. We believe that the condition placed on $x$ (that its star contains a hyperspecial vertex) is unnecessary; that is the content of Conjecture~\ref{conjecture: mackey} and the subject of current work.

Before concluding this introduction we acknowledge the elephant in the room: we do not know if the compact restriction $\cres{G}{x}\fais{F}$ of a character sheaf $\fais{F}$ of $\KKsch{G}$ is, in general, a semisimple perverse sheaf.  However, if $x_0$ is hyperspecial, then $\cres{G}{x_0}\fais{F} = \NC{\RRSch{G}{x_0}}{\fais{F}}$, which is perverse if $\fais{F}$ is perverse. In Proposition~\ref{proposition: BIRS} we show that more is true: if $\fais{F}$ is a character sheaf of $\KKsch{G}$ and $x_0$ is hyperspecial, then $\cres{G}{x_0}\fais{F}$ is a direct sum of character sheaves of $\rkkSch{G}{x_0} = \kkSch{G}{x_0}$, and therefore a semisimple perverse sheaf of geometric origin. This is a crucial ingredient in the proof of Theorem~\ref{theorem: 3bis}. 

We offer our thanks to Hadi Salmasian, who supplied several ideas for this paper.

\section{Cuspidal perverse sheaves}\label{section: pramod}

Let $\ff$ be any algebraically closed field and let $\ffsch{G}$ be any connected reductive linear algebraic group over $\ff$. A perverse sheaf $\fais{F}$ on $\ffsch{G}$ is a \emph{strongly cuspidal perverse sheaf}  \cite{lusztig:character-sheaves-II}*{7.1.5} if:
\begin{enumerate}
\item[(SC.1)] there is some $n\in \NN$ invertible in $\ff$ such that $\fais{F}$ is equivariant with respect to the $\ffsch{G}\times \mathcal{Z}_\ffsch{G}^\circ$ action on $\ffsch{G}$ defined by $(g,z) : h \mapsto z^n g h g^{-1}$;
\item[(SC.2)] $\res^\ffsch{G}_\ffsch{P}\fais{F} =0$ for every proper parabolic subgroup $\ffsch{P}\subset \ffsch{G}$.
\end{enumerate}
 A perverse sheaf $\fais{F}$ is a \emph{cuspidal perverse sheaf} if it satisfies an {\it a priori} weaker condition, articulated in \cite{lusztig:character-sheaves-I}*{7.1.1}; in particular, every strongly cuspidal perverse sheaf is a cuspidal perverse sheaf. In \cite{lusztig:character-sheaves-II}*{7.1.6}, Lusztig showed that every character sheaf is cuspidal if and only if it is strongly cuspidal. He also showed \cite{lusztig:character-sheaves-V}*{Theorem~23.1(b)} that if $\ffsch{G}$ is classical or exceptional in good characteristic, then every simple cuspidal perverse sheaf on $\ffsch{G}$ is a character sheaf. Accordingly, in these cases, it has been known for some time that every simple cuspidal perverse sheaf is a cuspidal character sheaf. More recently, Ostrik has made this result unconditional: every simple cuspidal perverse sheaf of $\ffsch{G}$ is a cuspidal character sheaf of $\ffsch{G}$, for every connected reductive linear algebraic group $\ffsch{G}$ over any algebraically closed field $\ff$ \cite{ostrik:cuspidal}*{Theorem~2.12}.

For every cuspidal character sheaf $\fais{F}$ there is a cuspidal pair $(\Sigma,\fais{E})$ \cite{lusztig:intersection-cohomology}*{Definition~2.4} such that $\fais{F} = j_{!*}\ \fais{E}[\dim \Sigma]$ \cite{lusztig:character-sheaves-I}*{Proposition~3.12} where $j: \Sigma \to \ffsch{G}$ is the inclusion of the locally closed subvariety $\Sigma$. The cuspidal character sheaf $j_{!*}\ \fais{E}[\dim \Sigma]$ is \emph{clean} 
\cite{lusztig:character-sheaves-II}*{Definition~7.7} if 
\[
j_{!*}\ \fais{E}[\dim \Sigma] \iso j_{!}\ \fais{E}[\dim \Sigma] \iso  j_{*}\ \fais{E}[\dim \Sigma].
\]

A connected, reductive linear algebraic group is \emph{clean} 
\cite{lusztig:character-sheaves-III}*{13.9.2} if every cuspidal character sheaf of every Levi subgroup of $\ffsch{G}$ is clean. Lusztig has conjectured that every connected reductive linear algebraic group $\ffsch{G}$ over an algebraically closed field $\ff$ is clean, and has shown \cite{lusztig:character-sheaves-V}*{Theorem~23.1~(a)} that if the characteristic of the field $\ff$ is not $2$, $3$ or $5$ then $\ffsch{G}$ is clean; in fact, \cite{lusztig:character-sheaves-V}*{Theorem~23.1~(a)} shows much more. This result was strengthened by Shoji in \cite{shoji:character-sheaves-I} and \cite{shoji:character-sheaves-II}, and again by Ostrik \cite{ostrik:cuspidal}*{Theorem~1}, in light of which we now know that if the characteristic of $\ff$ is not $2$ or if $\ffsch{G}$ has no factors of type $F_4$ or $E_8$, then $\ffsch{G}$ is clean. In particular, if the characteristic of $\ff$ is not $2$, then $\ffsch{G}$ is clean.

\begin{proposition}\label{proposition: Pramod}
If $\ffsch{G}$ is clean 
 then every strongly cuspidal perverse sheaf on $\ffsch{G}$ is a direct sum of cuspidal character sheaves; in particular, under these conditions every strongly cuspidal perverse sheaf on $\ffsch{G}$ is semisimple of geometric origin.
\end{proposition}

\begin{proof}
The category of perverse sheaves on $\ffsch{G}$ is Artinian and Noetherian: every perverse sheaf has finite length \cite{BBD}*{Th\'eor\`eme~4.3.1~(i)}. We prove Proposition~\ref{proposition: Pramod} by induction on the length (of the composition series) of cuspidal perverse sheaves on $\ffsch{G}$.  First, suppose $\fais{F}$ is a cuspidal perverse sheaf and the length of $\fais{F}$ is $1$. Then $\fais{F}$ is a simple perverse sheaf and a strongly cuspidal perverse sheaf, and therefore a cuspidal character sheaf, by hypothesis.

Next, suppose $\fais{F}$ is a strongly cuspidal perverse sheaf with length at least $2$. Let $\fais{G}$ be a simple sub-object of $\fais{F}$. Arguing as in \cite{lusztig:character-sheaves-I}*{1.9.1}, it follows that $\fais{G}$ satisfies condition SC.1 above (with $n$ determined by $\fais{F}$); in particular, $\fais{G}$ is an equivariant perverse sheaf. 

We will demonstrate that $\fais{G}$ satisfies condition SC.2. In the abelian category of perverse sheaves, form the short exact sequence below.
\begin{equation}\label{eqn: Pramod 1}
\xymatrix{
0 \ar[r] & \fais{G} \ar[r] & \fais{F} \ar[r] & \fais{F}/\fais{G} \ar[r] & 0 
}
\end{equation}
Then $\fais{F}/\fais{G}$ is equivariant (again, use \cite{lusztig:character-sheaves-I}*{1.9.1}).
Let $\ffsch{P}\subset \Ksch{G}$ be a proper parabolic subgroup; let $\ffsch{L}$ be its reductive quotient. Since $\res^\ffsch{G}_\ffsch{P} : \catM_\ffsch{G} \ffsch{G} \to \catM_\ffsch{L} \ffsch{L}$ is an exact functor (on and to equivariant perverse sheaves) it takes \eqref{eqn: Pramod 1} to the short exact sequence below.
\begin{equation}\label{eqn: Pramod 2}
\xymatrix{
0 \ar[r] & \res^\ffsch{G}_\ffsch{P}\fais{G} \ar[r] & \res^\ffsch{G}_\ffsch{P}\fais{F} \ar[r] & \res^\ffsch{G}_\ffsch{P}\left(\fais{F}/\fais{G}\right) \ar[r] & 0 
}
\end{equation}
Since $\fais{F}$ is a strongly cuspidal perverse sheaf (by hypothesis) and the sequence is exact, $\res^\ffsch{G}_\ffsch{P}\fais{G}=0$. Since $\ffsch{P}$ was an arbitrary proper parabolic subgroup of $\ffsch{G}$, it follows that $\fais{G}$ is a strongly cuspidal perverse sheaf. Since $\fais{G}$ is also simple, by hypothesis, it follows that $\fais{G}$ is a cuspidal character sheaf.

The paragraph above also shows that $\res^\ffsch{G}_\ffsch{P}\left(\fais{F}/\fais{G}\right) =0$ for every proper parabolic subgroup of $\ffsch{G}$. Thus, $\fais{F}/\fais{G}$ is a strongly cuspidal perverse sheaf. Since the length of $\fais{F}/\fais{G}$ is strictly less that that of $\fais{F}$, it follows (from the induction hypothesis) that $\fais{F}/\fais{G}$ is a direct sum of cuspidal character sheaves. Accordingly, we write $\fais{F}/\fais{G} = \oplus_{i\in I} \fais{G}_i$, where each $\fais{G}_i$ is a cuspidal character sheaf. 
Since $\fais{F}$ is an extension of $\fais{F}/\fais{G}$ by $\fais{G}$, it corresponds to an element of $\Ext^1(\fais{F}/\fais{G},\fais{G})$. Now, $\Ext^1(\fais{F}/\fais{G},\fais{G}) = \prod_{i\in I} \Ext^1(\fais{G}_i,\fais{G})$. Recall that $\fais{G}$ and each $\fais{G}_i$ are cuspidal character sheaves. 
It now follows from Lemma~\ref{lemma: Pramod} that $ \Ext^1(\fais{G}_i,\fais{G})=0$, and therefore that $\Ext^1(\fais{F}/\fais{G},\fais{G})=0$. This means that the short exact sequence in \eqref{eqn: Pramod 1} is split. Thus,
\[
\fais{F} = \fais{F}/\fais{G} \oplus \fais{G} = \mathop{\oplus}\limits_{i\in I} \fais{G}_i \oplus \fais{G}.
\]
Therefore, $\fais{F}$ is a direct sum of cuspidal character sheaves.
\end{proof}

\begin{lemma}\label{lemma: Pramod}
If $\ffsch{G}$ is clean then $\Ext^1(\fais{G}_i,\fais{G})=0$ for all cuspidal character sheaves $\fais{G}_i$, $\fais{G}$ of $\ffsch{G}$.
\end{lemma}

\begin{proof}

Let $\fais{G}$ (resp. $\fais{G}_i$) be a clean cuspidal character sheaf of $\ffsch{G}$. Then there is a unique cuspidal pair $(\Sigma,\fais{E})$  (resp. $(\Sigma_i,\fais{E}_i)$) such that $\fais{G} = \IC(\Sigma,\fais{E})[\dim \Sigma]$ (resp. $\fais{G}_i = \IC(\Sigma_i,\fais{E}_i)[\dim \Sigma_i]$). 
 Since we have assumed $\ffsch{G}$ is clean, $\fais{G}$ (resp. $\fais{G}_i$) is a clean cuspidal character sheaf. Since $\fais{G}$ (resp. $\fais{G}$) is clean, $\fais{G} = j_{!*}\ \fais{E}[\dim \Sigma] = j_{*}\ \fais{E}[\dim \Sigma]  = j _{!}\ \fais{E}[\dim \Sigma]$ (resp. $\fais{G}_i = (j_i)_{!*}\ \fais{E}_i[\dim \Sigma_i] = (j_i)_{*}\ \fais{E}_i[\dim \Sigma_i]  = (j_i)_{!}\ \fais{E}_i[\dim \Sigma_i]$).

Consider two cases. On the one hand, if $\Sigma \ne \Sigma_i$ then $\Sigma \cap \Sigma_i= \emptyset$ (this is a property of cuspidal pairs), and since $\fais{G}_i$ and $\fais{G}$ are clean, it follows that $\Ext^1(\fais{G}_i,\fais{G})=0$ for trivial reasons (they have disjoint support). 
On the other hand, suppose $\Sigma=\Sigma_i$. Then $\Ext^1(\fais{G}_i,\fais{G}) = \Ext^1(j_{!*}\fais{E}_i[\dim \Sigma], j_{!*}\fais{E}[\dim \Sigma])$. Since $\fais{G}_i$ and $\fais{G}$ are clean, 
\[
\Ext^1(j_{!*}\fais{E}_i[\dim \Sigma], j_{!*}\fais{E}[\dim \Sigma]) \iso\Ext^1(j_{!}\fais{E}_i, j_{*}\fais{E}).
\]
By adjunction,  
\[
\Ext^1(j_{!}\fais{E}_i, j_{*}\fais{E})= \Ext^1(j^{*}j_{!}\fais{E}_i, \fais{E}) =  \Ext^1(\fais{E}_i, \fais{E}).
\] 
Now $\fais{E}_i$ and $\fais{E}$ are local systems on $\Sigma$ corresponding (under an equivalence of categories determined by the choice of a geometric point $\bar{s}$ on $\Sigma$) to irreducible $\EE$-representations of the algebraic fundamental group $\pi_1(\Sigma, \bar{s})$. This group is compact (since it is profinite) and $\EE$ is algebraically closed of characteristic $0$, so the category of  $\EE$-representations of $\pi_1(\Sigma,\bar{s})$ is semisimple. Thus, $\Ext^1(\fais{E}_i, \fais{E})=0$. 
\end{proof}

\section{A little geometry}\label{subsection: geometry}

\begin{proposition}\label{proposition: geometry}
Let $\Ksch{G}$ be a connected, reductive linear algebraic group over a non-Archimedean local field $\Kq$. 
For every parabolic subgroup $\Ksch{P}\subseteq \Ksch{G}$ and for every $x\in I(\Ksch{L},\Kq)$ there is a smooth integral model $\RSch{P}{x}$ for $\Ksch{P}$ such that $\RSch{P}{x}(\Rq) = \Ksch{P}(\Kq) \cap \RSch{G}{x}(\Rq)$. Moreover, if $\Ksch{L}$ is the Levi subgroup of $\Ksch{P}$ and if $x$ actually lies in the building $I(\Ksch{L},\Kq)$ as a sub-building of $I(\Ksch{G},\Kq)$, then the quotient $\Ksch{\pi} : \Ksch{P} \to \Ksch{L}$ extends to a morphism of smooth integral models $\RSch{\pi}{x} :\RSch{P}{x} \to \RSch{L}{x}$, where $\RSch{L}{x}$ is the parahoric group scheme for $\Ksch{L}$ determined by $x$ as an element of $I(\Ksch{L},\Kq)$.
\end{proposition}

\begin{proof}
Let $\RSch{P}{x}$ be the schematic closure of $\Ksch{P}$ in $\RSch{G}{x}$. Observe that $\Ksch{P}$ is a closed subscheme of $\Ksch{G}$ and recall that, by definition (\cf \cite{bosch-lutkebohmert-raynaud:neron}*{\S 2.5}, for example), $\RSch{P}{x}$ is the smallest closed sub-scheme of $\RSch{G}{x}$ containing $\Ksch{P}$. By \cite{yu:models}*{\S 2.6, Lemma}, $\RSch{P}{x}$ is a model of $\Ksch{P}$ and $\RSch{P}{x}$ is a subscheme of $\RSch{G}{x}$. Let $\RSch{P}{x} \to \RSch{G}{x}$ be the closed immersion extending $\Ksch{P} \hookrightarrow \Ksch{G}$ such that $\RSch{P}{x}(\Rq) = \Ksch{P}(\Kq) \cap \RSch{G}{x}(\Rq)$ \cite{bruhat-tits:reductive-groups-2}*{1.7}.

Now we show that the $\Rq$-scheme $\RSch{P}{x}$ is smooth. Let $\Ksch{T}$ be a maximal torus of $\Ksch{G}$ contained in $\Ksch{L}$ and let $\Phi$ be the root system determined by the pair $(G,T)$. To simplify the exposition, we give the proof of smoothness for the case when $P$ is a Borel subgroup $B$ with Levi $\Ksch{T}$.

Without loss of generality, suppose $x$ lies in the apartment for $T$.
Let $\RSch{T}{}$ be the N\'eron-Raynaud model for $T$.
Arguing as in the proof of \cite{yu:models}*{\S 7, Theorem}, write $\RSch{B}{x}$ as $\RSch{T}{} \times \RSch{U}{x}$, where $\RSch{U}{x}$ is the image of $\prod_{\alpha} \RSch{U_\alpha}{x}$ under multiplication, where the product is taken over all roots in $\Phi$ that are positive for $\Ksch{B}$ and where $\RSch{U_\alpha}{x}$ is the unique smooth integral model of the root subgroup $\Ksch{U}_\alpha \subset \Ksch{G}$ such that  $\RSch{U_\alpha}{x}(\Rq) = \Ksch{U}_\alpha(\Kq)_{x,0}$ (\cf \cite{bruhat-tits:reductive-groups-2}*{\S 4.3}). 
Since $\RSch{T}{}$ and $\RSch{U}{x}$ are smooth, and since the product is taken over $\Spec{\Rq}$, it follows that $\RSch{B}{x}$ is also smooth.

The last point is clear.
\end{proof}


\section{$\Gm$-equivariant base change}\label{section: Braden base change}

Notation from Proposition~\ref{proposition: geometry}.

\begin{proposition}\label{proposition: res cres}
Let $\Ksch{P}$ be a parabolic subgroup of $\Ksch{P}$ with Levi subgroup $\Ksch{L}$.  Suppose $x_0\in I(\Ksch{L},\Kq)\hookrightarrow I(\Ksch{G},\Kq)$ is hyperspecial. For every equivariant perverse sheaf $\fais{G}$ on $\KKsch{G}$,
\[
\cres{L}{x_0}\ \res^\KKsch{G}_\KKsch{P} \fais{G} 
\iso \res^\kkSch{G}{x_0}_\kkSch{P}{x_0} \ \cres{G}{x_0} \fais{G}.
\]
\end{proposition}

\begin{proof}
\[
\begin{aligned}
&\cres{L}{x_0}\ \res^\KKsch{G}_\KKsch{P} \fais{G} \\
&:= \NC{\RRSch{L}{x_0}} (\pi_\KKsch{P})_{!}\ \left( \fais{G}\vert_\KKsch{P}\right) \\
&\iso (\kkSch{\pi}{x_0})_{!}\ \NC{\RRSch{P}{x_0}} \ \left( \fais{G}\vert_\KKsch{P}\right)  \hskip1cm \text{(Lemma~\ref{lemma: Braden base change})} \\
&\iso (\kkSch{\pi}{x_0})_{!}\  \left( \NC{\RRSch{L}{x_0}} \ \fais{G}\right)\vert_\kkSch{P}{x_0} \hskip1cm \text{(smooth base change)}\\
&=: \res^\kkSch{G}{x_0}_\kkSch{P}{x_0} \ \cres{G}{x_0} \fais{G}\\
\end{aligned}
\]
\end{proof}

\begin{lemma}\label{lemma: Braden base change}
Let $\Ksch{P}$ be a parabolic subgroup of $\Ksch{G}$ with Levi subgroup $\Ksch{L}$.  Suppose $x\in I(\Ksch{L},\Kq)\hookrightarrow I(\Ksch{G},\Kq)$. If $\fais{F}$ is an equivariant perverse sheaf on $\KKsch{P}$ then there is a canonical isomorphism
\[
\NC{\RRSch{L}{x}} (\KKSch{\pi}{x})_{!} \fais{F}
\iso
(\kkSch{\pi}{x})_{!}\ \NC{\RRSch{P}{x}}  \fais{F}.
\]
\end{lemma}

\begin{proof}
The quotient $\pi : P \to L$ is not proper, so this is not an instance of proper base change. Instead, we must do some work.
The proof of Lemma~\ref{lemma: Braden base change} is obtained by introducing an action of $\mathbb{G}_{m,\Rq}$ on $\RSch{P}{x}$ and then adapting results from \cite{braden:hyperbolic}*{Lemma~6} and \cite{springer:purity}*{Corollary~1}. Appendix~\ref{appendix: toric schemes} explains how the terms `invariant-theoretic quotient' and `contracting' are used below.

In order the use \cite{braden:hyperbolic}, we must establish the following facts.
\begin{enumerate}
\item[(B.1)] $\Ksch{\pi} : \Ksch{P} \to \Ksch{L}$ is the invariant-theoretic quotient of a contracting $\Kq$-action of $\Gm_{,\Kq}$ on $\Ksch{P}$;
\item[(B.2)] $\fais{F}$ is equivariant for the action of $\Gm_{,\KKq}$ on $\KKsch{P}$ obtained by extension of scalars;
\item[(B.3)]
$\RSch{\pi}{x} : \RSch{P}{x} \to \RSch{L}{x}$ is the invariant-theoretic quotient of a contracting $\Rq$-action of $\Gm_{,\Rq}$ on $\RSch{P}{x}$;
\item[(B.4)] $(j_\RRSch{P}{x})_*\ \fais{F}$ is equivariant for the action of  $\Gm_{,\RRq}$ on $\RRSch{P}{x}$ obtained by extension of scalars.
\end{enumerate}
Although B.1 and B.2 actually follow from B.3 and B.4, we begin by explaining what B.1 and B.2 mean and how to use them to prove part of this lemma, before moving on to the more complicated statements B.3 and B.4 and how to use them. To simplify the exposition, here we only treat the case when $P$ is a Borel subgroup.  

As in the proof of Proposition~\ref{proposition: geometry}, let $T$ be a maximal torus of $G$ contained in $L$ and let $\Phi$ be the root system determined by the pair $(G,T)$. Let $B$ be a Borel subgroup of $G$ with Levi $T$; let $U \supseteq U_P$ be the unipotent radical of $B$. Let $\Phi^+$ be the set of roots in $\Phi$ that are positive for $B$. Let $\delta$ be the character of $Z_T=T$ defined by $\delta = \frac{1}{2}\sum_{\alpha\in\Phi^+} \alpha$; this is a (strongly) dominant weight and the co-character $\check{\delta}$ is a dominant cocharacter. Then $\mu: \Gm_{,\Kq}\times B \to B$, defined by $\mu : (z,b) \mapsto\check{\delta}(z)\, b\, \check{\delta}(z)^{-1}$, is an action of $\Gm_{,\Kq}$ on $B$. Moreover, because $\check{\delta}$ centralizes $T$, the restriction of $\mu$ to $\Gm_{,\Kq}\times U$ defines an action $\mu_{U}$ of $\Gm_{,\Kq}$ on $U$ such that $\mu(z,u\rtimes t) = \mu_{U}(z,u)\rtimes t$, with reference to $B \iso U\rtimes T$. Accordingly, $B^{\Gm_{,\Kq}}\iso U^{\Gm_{,\Kq}} \rtimes T$. Using the classical monomorphisms $u_\alpha : \Gadd_{,\Kq} \to U$ with image $U_\alpha$ (see \cite{springer:lag2} for example) and their fundamental properties (see 
 or \cite{springer-steinberg:conj}*{1.1, 1.2(b)}, for example), it follows that $U_B^{\Gm_{,\Kq}} = 1$.  (The main point here is $U \iso \prod_{\alpha \in \Phi^+} U_\alpha$ and $t u_\alpha(\xi) t^{-1} = u_\alpha( \alpha(t) \xi)$.) Thus, $B^{\Gm_{,\Kq}} = T$. Moreover, it follows from these same fundamental facts that there is a map $\bar\mu: \mathbb{A}^1_\Kq \times B \to B$ such that the following diagram commutes,
\[
\xymatrix{
\Gm_{,\Kq} \times B \ar[rr]^{\mu} \ar[dr] && B \\
& \mathbb{A}^1_\Kq \times B \ar[ur]_{\bar\mu} 
}
\]
where the bottom-left arrow is the open subscheme in the first component and the identity on the second.  Since this is exactly what it means to say that the action of $\Gm_{,\Kq}$ on $B$ is contracting, and since $B^\Gm = T$, it follows automatically that $T = B \git \Gm_{,\Kq}$ (\cf Appendix~\ref{appendix: toric schemes}) and that $\pi : B \to T$ is the invariant-theoretic quotient (\cf Appendix~\ref{appendix: toric schemes}). Thus,  B.1 is established.

Extending scalars from $\Kq$ to $\KKq$ gives an $\KKq$-action of $\Gm_{,\KKq}$ on $\KKsch{B}$ which again is contracting. Because $\fais{F}$ is an equivariant perverse sheaf for conjugation, and because the $\KKq$-action of $\Gm_{,\KKq}$ on $\KKsch{B}$ is defined by conjugation by the co-character $\check{\delta}$, $\fais{F}$ is equivariant for this action also. This establishes B.2. 

We now explain the significance of facts B.1 and B.2. Let $\KKsch{\iota} : \KKsch{L} \hookrightarrow \KKsch{P}$ be inclusion. This is a section of $\KKsch{\pi} : \KKsch{P} \to \KKsch{L}$. Thus, $\KKsch{\pi} \circ \KKsch{\iota}$ is the identity morphism of $\KKsch{L}$, so $\KKsch{\pi}_*\ \KKsch{\iota}_*$ is isomorphic to the identity functor. Composing $\KKsch{\pi}_*$ with the adjunction morphism $\id \to \KKsch{\iota}_*\ \KKsch{\iota}^*$, we thereby obtain a morphism of functors $\KKsch{\pi}_* \to \KKsch{\iota}^*$ and, dually, $\KKsch{\iota}^! \to \KKsch{\pi}_!$. Now, because of B.1 and B.2, \cite{braden:hyperbolic}*{\S 6} applies and shows that these morphisms of functors induce isomorphisms on $\Gm_{,\KKq}$-equivariant sheaves! In particular, 
\begin{equation}\label{equation: Braden base change 1}
\KKsch{\iota}^!\fais{F} \iso \KKsch{\pi}_!\fais{F}
\end{equation}
Accordingly,
\begin{equation}\label{equation: Braden base change 2}
 \NC{\RRSch{L}{x}} \KKsch{\pi}_{!}\ \fais{F} \iso (i_\RRSch{L}{x})^*\ (j_\RRSch{L}{x})_*\ \KKsch{\iota}^{!}\ \fais{F}.
 \end{equation}

With reference to notation from Proposition~\ref{proposition: geometry}, let $\RSch{\iota}{x} : \RSch{L}{x} \hookrightarrow \RSch{P}{x}$ be inclusion. Then $\KKSch{\iota}{x} = \KKsch{\iota}$ and, by base change,
\begin{equation}\label{equation: Braden base change 3}
(j_\RRSch{L}{x})_*\ (\KKSch{\iota}{x})^{!}\ \fais{F}
\iso 
 (\RRSch{\iota}{x})^!\ (j_\RRSch{P}{x})_*\ \fais{F}.
 \end{equation}

Arguing as above, we obtain a morphism of functors $(\RRSch{\iota}{x})^! \to (\RRSch{\pi}{x})_!$. We will see that this is an isomorphism of functors on the appropriate category of sheaves. To do this, we turn to B.3 and B.4.

In order to prove B.3 and B.4 we may suppose, without loss of generality, that $x_0$ lies in the apartment determined by $T$. Set $\RSch{T}{}\ceq \RSch{T}{x}$; this is the N\'eron-Raynaud model for $\Ksch{T}$. Each co-character of $T$ extends to a $\Kq$-morphism $\Gm_{,\Rq} \to \RSch{T}{}$, as in the proof of \cite{cunningham-salmasian:sheaves}*{Proposition~4}. Using \cite{cunningham-salmasian:sheaves}*{\S 2.3}, define an $\Rq$-action $\RSch{\mu}{x} : \Gm_{,\Rq}\times \RSch{B}{x} \to \RSch{B}{x}$ as above (using the extension of the dominant co-character $\check{\delta}$). To see that this action is contracting (\cf Appendix~\ref{appendix: toric schemes}) recall, for each $\alpha \in \Phi$, the smooth integral scheme $\RSch{U_\alpha}{x}$ for $U_\alpha$ that appeared in the proof of Proposition~\ref{proposition: geometry}, and also $\RSch{U}{x}$. Each $\RSch{U_\alpha}{x}$ comes equipped with a $\Kq$-morphism $\RSch{u_\alpha}{x} : \Gadd \to \RSch{U_\alpha}{x}$ that satisfies the analogue of \cite{springer-steinberg:conj}*{1.1} and $\RSch{U}{x}$ satisfies the analogue \cite{springer-steinberg:conj}*{1.2(b)} with regards to the additive schemes $\RSch{U_\alpha}{x}$! 
\[
\xymatrix{
\Gm_{,\Rq} \times \RSch{B}{x} \ar[rr]^{\RSch{\mu}{x}} \ar[dr] && \RSch{B}{x} \\
& \mathbb{A}^1_\Rq \times \RSch{B}{x} \ar[ur]_{\RSch{\bar\mu}{x}} 
}
\]
The proof of B.1, above, adapts to the present context, and gives B.3.
The fact that $j_\RSch{P}{x}$ is morphism of group schemes gives B.4.

The final miracle is that the proof in \cite{braden:hyperbolic}*{\S 6}, which is largely formal, applies to the category of $\RRq$-schemes.
Accordingly, facts B.3 and B.4 determine
\begin{equation}\label{equation: Braden base change 4}
(\RRSch{\iota}{x})^!\ (j_\RRSch{P}{x})_*\ \fais{F} \iso (\RRSch{\pi}{x})_!\ (j_\RRSch{P}{x})_*\ \fais{F}.
\end{equation}
Thus,
\begin{equation}\label{equation: Braden base change 5}
(i_\RRSch{L}{x})^*\ (\RRSch{\iota}{x})^!\ (j_\RRSch{P}{x})_*\ \fais{F}
\iso
(i_\RRSch{L}{x})^*\ (\RRSch{\pi}{x})_!\ (j_\RRSch{P}{x})_*\ \fais{F}.
\end{equation}
By base change,
\begin{equation}\label{equation: Braden base change 6}
(i_\RRSch{L}{x})^*\ (\RRSch{\pi}{x})_!\ (j_\RRSch{P}{x})_*\ \fais{F}
\iso
(\kkSch{\pi}{x})_{!} \ (i_\RRSch{P}{x})^*\ (j_\RRSch{P}{x})_*\ \fais{F},
\end{equation}
and by the definition of the nearby cycles functor,
\begin{equation}\label{equation: Braden base change 7}
(\kkSch{\pi}{x})_{!} \ (i_\RRSch{P}{x})^*\ (j_\RRSch{P}{x})_*\ \fais{F}
\iso (\kkSch{\pi}{x})_{!}\ \NC{\RRSch{P}{x}} \ \fais{F}.
\end{equation}
Combining Equations \eqref{equation: Braden base change 2},  \eqref{equation: Braden base change 3}, \eqref{equation: Braden base change 5}, \eqref{equation: Braden base change 6} and \eqref{equation: Braden base change 7} gives the proof of Lemma~\ref{lemma: Braden base change}.
\end{proof}

\section{Nearby cycles of cuspidal character sheaves}

\begin{proposition}\label{proposition: nc of ccs}
Suppose $\Ksch{G}$ is a connected, reductive linear algebraic group over a non-Archimedean local field $\Kq$. If $\RSch{G}{x_0}$ is hyperspecial and $\fais{G}$ is a cuspidal character sheaf of $\KKsch{G}$ then $\cres{G}{x_0}\fais{G}$ is a strongly cuspidal perverse sheaf on $\kkSch{G}{x_0}$.
\end{proposition}

\begin{proof}
First we show that $\cres{G}{x_0}\fais{G}$ is a strongly cuspidal perverse sheaf.
Let $\kksch{Q}$ be a proper parabolic subgroup of $\rkkSch{G}{x_0}$; let $\kksch{M}$ be the Levi subgroup of $\kksch{Q}$. We will see that $\res^\rkkSch{G}{x_0}_\kksch{Q} \cres{G}{x_0}\fais{G} =0$.  

The parabolic subgroup $\kksch{Q}$ is defined over some finite extension $\lq$ of $\kq$, so we write $\kksch{Q} = \sch{Q}\times_\Spec{\lq} \Spec{\kkq}$ where $\sch{Q}$ is a linear algebraic group over $\lq$. Let $\sch{M}$ be the reductive quotient of $\sch{Q}$. Let $\Lq$ be the unramified extension of $\Kq$ in $\KKq$ with residue field $\lq$. Let $x_0'$ denote the image of $x_0$ under $I(\Ksch{G},\Kq) \hookrightarrow I(\Ksch{G},\Lq)$ and let $\RSch{G}{x_0'}$ be the parahoric group scheme for $\Lsch{G}$ determined by $x_0'$. (Since $\Lq/\Kq$ is unramified, $\RSch{G}{x_0'} = \RSch{G}{x_0}\times_\Spec{\Rq}\Spec{\RLq}$.
Pick  $x'\in I(\Ksch{G},\Lq)$ such that $x' > x_0'$ and $\rkPSch{P}{x_0'\leq x'} = \sch{Q}$ and $\rlSch{G}{x'} = M$, where $\rkPSch{P}{x_0'\leq x'}$ is as defined in \cite{cunningham-salmasian:sheaves}*{\S 2.1} (where it is denoted by $\rkPSch{P}{x\leq y}$). (Such an $x'$ can be found because, locally, the affine building at $x_0'$ corresponds to the (spherical) building for $\rlSch{G}{x_0'}$ \cite{landvogt:compactification}.)
Then
\[
\res^\kkSch{G}{x_0}_\kksch{Q}\ \cres{G}{x_0}\fais{G}
=
\res^\kkSch{G}{x_0}_\kksch{Q}\ \cres{G}{x_0'}\fais{G}
=
\res^\kkSch{G}{x_0'}_\rkkPSch{Q}{x_0'\leq x'}\ \cres{G}{x_0'}\fais{G}.
 \]
On the other hand, the relative position of $x_0'$ and $x'$ in $I(\Ksch{G},\Lq)$ also determines a proper parabolic subgroup $\sch{P}$ of $\Lsch{G}$ such that $\rlSch{P}{x_0'} = \rkPSch{P}{x_0'\leq x'}$.  Thus,
\[
\res^\kkSch{G}{x_0'}_\rkkPSch{P}{x_0'\leq x'}\ \cres{G}{x_0'}\fais{G}
=
\res^\kkSch{G}{x_0'}_\kkSch{P}{x_0'}\ \cres{G}{x_0'}\fais{G}.
\]
It follows from  Proposition~\ref{proposition: res cres} (with $\Kq$ replaced by $\Lq$) that,
\[
\res^\kkSch{G}{x_0'}_\kkSch{P}{x_0'}\ \cres{G}{x_0'}\fais{G}
=
\cres{L}{x_0'}\ \res^\KKsch{G}_\KKsch{P} \fais{G},
\]
where $\Ksch{L}$ is the Levi subgroup of $\Ksch{P}$. (Observe that $x_0'$ lies in the image of $I(\Ksch{L},\Lq) \hookrightarrow I(\Ksch{G},\Lq)$, by design.) We have used the fact that $\fais{G}$ is a cuspidal character sheaf since it is strongly cuspidal. Since $P$ is a proper parabolic subgroup of $G$, it follows that $\res^\KKsch{G}_\KKsch{P} \fais{G}=0$. 
We have now seen that $\res^\rkkSch{G}{x_0}_\kksch{Q}\ \cres{G}{x_0}\fais{G}=0$ for every proper parabolic subgroup $\kksch{Q}$ of $\rkkSch{G}{x_0}$. Thus, $\cres{G}{x_0}\fais{G}$ satisfies condition SC.2 (\cf Section~\ref{section: pramod}). To verify condition SC.1 one uses \cite{lusztig:character-sheaves-I}*{1.9.1}, as in the proof of Proposition~\ref{proposition: Pramod}.
\end{proof}

\section{Nearby cycles of cuspidal character sheaves are semisimple}

\begin{proposition}\label{proposition: semisimple}
Suppose $\Ksch{G}$ is a connected, reductive linear algebraic group over non-Archimedean local field $\Kq$ of odd or zero characteristic. If $\RSch{G}{x_0}$ is hyperspecial and $\fais{G}$ is a cuspidal character sheaf of $\KKsch{G}$ then $\cres{G}{x_0}\fais{G}$ is a direct sum of cuspidal character sheaves on $\kkSch{G}{x_0}$.
\end{proposition}
 
\begin{proof}
By proposition~\ref{proposition: nc of ccs}, $\cres{G}{x_0}\fais{G}$ is a strongly cuspidal perverse sheaf on $\rkkSch{G}{x_0}$. Since the characteristic of $\Kq$ is not $2$, the residual characteristic of $\Kq$ is odd. Accordingly, $\rkkSch{G}{x_0}$ is clean (\cite{ostrik:cuspidal}*{Theorem~1}, improving\cite{lusztig:character-sheaves-V}*{Theorem~23.1~(a)}) and every simple cuspidal perverse sheaf on $\kkSch{G}{x_0}$ is a character sheaf (\cite{ostrik:cuspidal}*{Theorem~2.12} improving \cite{lusztig:character-sheaves-V}*{Theorem~23.1~(b)}). It now follows from Proposition~\ref{proposition: Pramod} that $\cres{G}{x_0}\fais{G}$ is a direct sum of cuspidal character sheaves.
\end{proof}

\begin{remark}\label{remark: semisimple}
If $\fais{G}$ is a cuspidal character sheaf of $\KKsch{G}$ and $x\in I(\Ksch{G},\Kq)$ is not hyperspecial and the star of $x$ contains a hyperspecial vertex, then $\cres{G}{x}\fais{G} =0$. This follows from the proof of Proposition~\ref{proposition: res cres} and \cite{cunningham-salmasian:sheaves}*{Theorem~1}. We will not use that fact in this paper.
\end{remark}

\section{A little more geometry}\label{subsection: more geometry}

\begin{proposition}\label{proposition: more geometry}
Let $\Ksch{G}$ be a connected, reductive linear algebraic group over a non-Archimedean local field $\Kq$. 
For every parabolic subgroup $\Ksch{P}\subseteq \Ksch{G}$ and every $x \in I(\Ksch{G},\Kq)$ there is a smooth integral model $\RSch{G}{x}/\RSch{P}{x}$ for $\Ksch{G}/\Ksch{P}$, and a principal fibration $\RSch{G}{x} \to \RSch{G}{x}/\RSch{P}{x}$ with group $\RSch{P}{x}$ such that the special fibre of $\RSch{G}{x}/\RSch{P}{x}$ is the quotient variety $\kSch{G}{x}/\kSch{P}{x}$.
\end{proposition}

\begin{proof}
To simplify the exposition we replace $\Ksch{P}$ by a Borel subgroup $\LBsch{B}$ and construct $\RSch{G}{x} \to \RSch{G}{x}/\RSch{B}{x}$. Standard techniques extend this construction to give $\RSch{G}{x} \to \RSch{G}{x}/\RSch{P}{x}$.

We construct $\RSch{G}{x}/\RSch{B}{x}$ and the fibration $\RSch{G}{x}\to \RSch{G}{x}/\RSch{B}{x}$. 
With $\Phi$ as in the proof of Proposition~\ref{proposition: geometry}, let $\Phi_{x}$ (resp. $\Phi_{x}^+$) be the set of roots $\alpha\in \Phi$ (resp. $\alpha\in \Phi^+$) for which $\vec{\alpha}(x)=0$, where $\vec{\alpha}$ is an affine root of $\Ksch{G}$ with vector part equal to $\alpha$. Also, let $W_{x}$ be the Weyl group for the root system $\Phi_{x}$. For each $w\in W_{x}$, define $\Phi_{x}(w)^+ \ceq \{ \alpha \in \Phi_{x}^+ \tq w(\alpha) \in \Phi_{x}^-\}$.
The image of $\prod_{\alpha\in \Phi_{x}(w)^+} \RSch{U_\alpha}{x}$ under the multiplication map to $\RSch{G}{x}$ will be denoted by $\RSch{U_w}{x}$. Let $\RSch{G_w}{x} \subset \RSch{G}{x}$ be the (locally closed) subscheme $\RSch{U_w}{x} \dot{w} \RSch{B}{x}$, where $\dot{w} \in \RSch{G}{x}(\Rq)$ is a representative for $w$. Then $\RSch{U_w}{x}$ is isomorphic to $\AA^{l(w)}_\sch{S}$ and $\RSch{G_w}{x}$ is isomorphic to $\AA^{l(w)}_\sch{S} \times \RSch{B}{x}$. Let $w_0$ be the Coxeter element in $W_{x}$ (recall that $\Phi_{x}$ is a reduced root system). Then $\RSch{G_{w_0}}{x}\subset \RSch{G}{x}$ is an open subscheme and $\RSch{G}{x} = \mathop{\cup}\limits_{w\in W_{x}} \dot{w} \RSch{G_{w_0}}{x} \dot{w}^{-1}$ is an open covering.

We can now define $\RSch{G}{x}/\RSch{B}{x}$ by gluing data, as follows. For each $w\in W_{x}$, let $\RSch{b(w)}{x}: \dot{w}\RSch{G_{w_0}}{x} \dot{w}^{-1} \to \AA^{l(w_0)}_\Rq$ be the obvious map (conjugate to $\RSch{G_{w_0}}{x}$, then use $\RSch{G_{w_0}}{x}\iso \AA^{l(w_0)}_\Rq \times \RSch{B}{x}$ and finally project to $\AA^{l(w_0)}_\Rq$). For each pair $w_1, w_2\in W_{x}$, set $V_{w_1} = \AA^{l(w_0)}_\Rq$; also, let $V_{w_1,w_2}$ be the image of $\dot{w_1}\RSch{G_{w_0}}{x} \dot{w_1}^{-1} \cap \dot{w_2}\RSch{G_{w_0}}{x} \dot{w_2}^{-1}$ under $\RSch{b(w_1)}{x} : \dot{w_1}\RSch{G_{w_0}}{x} \dot{w_1}^{-1} \to \AA^{l(w_0)}_\Rq$. For each pair $w_1, w_2\in W_{x}$, 
glue $V_{w_1}$ to $V_{w_2}$ along $V_{w_1,w_2} \iso V_{w_2,w_1}$. The resulting scheme is $\RSch{G}{x}/\RSch{B}{x}$.  

We have now defined $\RSch{G}{x}/\RSch{B}{x}$ and also $\RSch{b}{x} : \RSch{G}{x}\to \RSch{G}{x}/\RSch{B}{x}$. It is clear that $\RSch{b}{x}$ is a principal fibration with group $\RSch{B}{x}$. Since this fibration is given locally by $\RSch{b(w)}{x}$ --- which is defined by composing two isomorphisms and then projecting $\AA^{l(w_0)}_\Rq \times \RSch{B}{x}$ --- the fibration is smooth.

A smooth fibration $\RSch{p}{x} : \RSch{G}{x}\to \RSch{G}{x}/\RSch{P}{x}$ with group $\RSch{P}{x}$ is defined by similar arguments. 
From the construction above we see that the special fibre of $\RSch{G}{x} \to \RSch{G}{x}/\RSch{P}{x}$ is a cokernel of $\kSch{P}{x}\to \kSch{G}{x}$ in the category of algebraic varieties over $\kq$.
\end{proof}

\begin{figure}[htbp]
\begin{center}
\[
\xymatrix{
\Ksch{P} \ar[r] & \Ksch{G} \ar[r] & \Ksch{G}/\Ksch{P} & \\
\KSch{P}{x} \ar@{=}[u]\ar[r] \ar[d] & \KSch{G}{x} \ar[d]_{j_\RSch{G}{x}} \ar@{=}[u]\ar[r] & \KSch{G}{x}/\KSch{P}{x} \ar@{=}[u] \ar[d] \ar[r] & \Spec{\Kq} \ar[d] \\
\RSch{P}{x} \ar[r]  & \RSch{G}{x} \ar[r]  & \RSch{G}{x}/\RSch{P}{x} \ar[r] & \Spec{\Rq} \\
\kSch{P}{x} \ar[r] \ar[d]_{\nu_\RSch{G}{x}\vert_{\kSch{P}{x}}} \ar[u] & \kSch{G}{x} \ar[r] \ar[d]^{\nu_\RSch{G}{x}} \ar[u]^{i_\RSch{G}{x}} & \kSch{G}{x}/\kSch{P}{x} \ar[d] \ar[u] \ar[r] & \Spec{\kq} \ar[u] \\
\rk{G}{P}{x} \ar[r] & \rkSch{G}{x} \ar[r] & \rkSch{G}{x}/\rk{G}{P}{x} \ar[ru] & \\
}
\]
\caption{The quotient scheme $\RSch{G}{x}/\RSch{P}{x}$}
\label{diagramme: quotient}
\end{center}
\end{figure}

\begin{remark}
If $x_0$ is hyperspecial then $\RSch{G}{x_0}/\RSch{P}{x_0}$ is projective. We will not use that fact in this paper.
\end{remark}

\section{A (hyper)special case of the Mackey formula}

\begin{proposition}\label{proposition: cres ind}
Let $\Ksch{G}$ be a connected reductive linear algebraic group over a non-Achimedean local field $\Kq$ of odd or zero characteristic. 
Let $\Lq/\Kq$ be a finite unramified extension.
Let $\LBsch{P}$ be a parabolic subgroup of $\Ksch{G}\times_\Spec{\Kq} \Spec{\Lq}$ with reductive quotient $\LBsch{L}$. Suppose $x_0\in I(\Ksch{G},\Kq)$ is hyperspecial and that the image $x_0'$ of $x_0$ under $I(\Ksch{G},\Kq)\hookrightarrow I(\Ksch{G},\Lq)$ also lies in the image of $I(\Ksch{L},\Lq)\hookrightarrow I(\Ksch{G},\Lq)$. For every equivariant perverse sheaf $\fais{G}$ on $\KKsch{L}$,
\[
\cres{G}{x_0}\ \ind^\KKsch{G}_\KKsch{P} \fais{G}
\iso \ind^\kkSch{G}{x_0}_\kkSch{P}{x_0'} \ \cres{L}{x_0'} \fais{G},
\]
where $\RLBSch{P}{x_0'}$ is the smooth $\RLq$-scheme introduced in Proposition~\ref{proposition: more geometry}.
\end{proposition}

\begin{proof}
The proof of Proposition~\ref{proposition: cres ind} follows the argument for \cite{cunningham-salmasian:sheaves}*{Theorem~3}, with small adaptations, which we include here.
We write $x_0'$ for the image of $x_0$ under $I(\Ksch{L},\Kq) \to I(\Ksch{L},\Lq)$ and under  $I(\Ksch{G},\Kq) \to I(\Ksch{G},\Lq)$.
Consider the $\RLq$-schemes
\begin{equation*}
\begin{aligned}
&\RLBSch{X}{x_0'} \ceq \left\{ (g,h) \in \RLBSch{G}{x_0'}\times\RLBSch{G}{x_0'} \tq h^{-1}gh\in \RLBSch{P}{x_0'}\right\} \iso \RLBSch{G}{x_0'}\times \RLBSch{P}{x_0'} \\
&\RLBSch{Y}{x_0'} \ceq\left\{ (g,h\RLBSch{P}{x_0'}) \in \RLBSch{G}{x_0'}\times\left(\RLBSch{G}{x_0'}/\RLBSch{P}{x_0'}\right) \tq h^{-1}gh\in \RLBSch{P}{x_0'}\right\}.
\end{aligned}
\end{equation*}
By Proposition~\ref{proposition: more geometry}, these are smooth schemes and the morphism $\RLBSch{\beta}{x_0'} : \RLBSch{X}{x_0'} \to  \RLBSch{Y}{x_0'}$ defined by $\RLBSch{\beta}{x_0'}(g,h) \ceq (g,h\RLBSch{P}{x_0'})$ is a $\RLBSch{P}{x_0'}$-torsor.  It also follows from Proposition~\ref{proposition: more geometry} (with the field $\Kq$ replaced by $\Lq$) that the generic fibres of $\RLBSch{X}{x_0'}$ and $\RLBSch{Y}{x_0'}$ are the classical varieties
\begin{equation*}
\begin{aligned}
\LBSch{X}{x_0'} \iso & X_P \ceq \left\{ (g,h) \in \Lsch{G}\times \Lsch{G} \tq h^{-1}gh\in P\right\} \iso \Lsch{G}\times P \\ 
\LBSch{Y}{x_0'} \iso & Y_P \ceq\left\{ (g,hP) \in \Lsch{G} \times\left(\Lsch{G}/P\right) \tq h^{-1}gh\in P\right\} 
\end{aligned}
\end{equation*}
The generic fibre $\LBSch{\beta}{x_0'}$ of $\RLBSch{\beta}{x_0'}$ is the smooth principal $P$- fibration $\beta_P : X_P \to Y_P$ defined by $\beta_P(g,h) \ceq (g, hP)$.
Using Proposition~\ref{proposition: more geometry} we find that the special fibres $\lBSch{X}{x_0'}$ and $\lBSch{Y}{x_0'}$ are 
\begin{equation*}
\begin{aligned}
&\lBSch{X}{x_0'} \ceq \left\{ (g,h) \in \lBSch{G}{x_0'}\times\lBSch{G}{x_0'} \tq h^{-1}gh\in \lBSch{P}{x_0'}\right\} \iso \lBSch{G}{x_0'}\times \lBSch{P}{x_0'} \\
&\lBSch{Y}{x_0'} \ceq\left\{ (g,h\lBSch{P}{x_0'}) \in \lBSch{G}{x_0'}\times\left(\lBSch{G}{x_0'}/\lBSch{P}{x_0'}\right) \tq h^{-1}gh\in \lBSch{P}{x_0'}\right\} 
\end{aligned}
\end{equation*}
and that the special fibre $\lBSch{\beta}{x_0'}$ of $\RLBSch{\beta}{x_0'}$ is the smooth principal fibration $\beta_\lBSch{P}{x_0'} : X_\lBSch{P}{x_0'} \to Y_\lBSch{P}{x_0'}$ defined by $\beta_\lBSch{P}{x_0'}(g,h) \ceq (g, h\lBSch{P}{x_0'})$.

Again, with reference to Proposition~\ref{proposition: more geometry}, let $\RLBSch{\pi}{x_0'} : \RLBSch{P}{x_0'} \to \RLBSch{L}{x_0'}$ be the extension of the reductive quotient map $\LBsch{\pi_P} : \LBsch{P} \to \Lsch{L}$ (existence and uniqueness is given by the Extension Principle, as in \cite{yu:models}*{2.3}, for example) and define $\RLBSch{\alpha}{x_0'} : \RLBSch{X}{x_0'} \to \RLBSch{L}{x_0'}$ by $\RLBSch{\alpha}{x_0'}(h,p) = \RLBSch{\pi}{x_0'}(h^{-1}gh)$. We remark that $\RLBSch{\alpha}{x_0'}$  is smooth. The generic fibre of $\RLBSch{\alpha}{x_0'}$ is $\alpha_P(g,h) = \pi_P(h^{-1}gh)$; the special fibre of $\RLBSch{\alpha}{x_0'}$ is defined likewise by $\lBSch{\alpha}{x_0'}(g,h) = \lBSch{\pi}{x_0'}(h^{-1}gh)$.

\begin{figure}[htbp]
\begin{center}
\[
	\xymatrix{
	\Lsch{G} \ar@{=}[r] & \LBSch{G}{x_0'} \ar[r]^{j_{\RLBSch{G}{x_0'}}}& \RLBSch{G}{x_0'} & \ar[l]_{i_{\RLBSch{G}{x_0'}}} \lBSch{G}{x_0'} \ar@{=}[r] & \lBSch{G}{x_0'} \\
	\ar[u]_{\proj_1} \sch{Y}_{\LBsch{P}} \ar@{=}[r] & \ar[u]_{\proj_1} \LBSch{Y}{x_0'} \ar[r]^{j_{\RLBSch{Y}{x_0'}}}& \ar[u]_{\proj_1} \RLBSch{Y}{x_0'}  & \ar[l]_{i_{\RLBSch{Y}{x_0'}}} \ar[u]_{\proj_1} \lBSch{Y}{x_0'} \ar@{=}[r] & \ar[u]_{\proj_1} \sch{Y}_{\lBSch{P}{x_0'}} \\
	\ar[d]^{\alpha_{\LBsch{P}}} \ar[u]_{\beta_{\LBsch{P}}} \sch{X}_{\LBsch{P}} \ar@{=}[r] & \ar[u]_{\LBSch{\beta}{x_0'}} \ar[d]^{\KSch{\alpha}{x_0}} \LBSch{X}{x_0'} \ar[r]^{j_{\RLBSch{X}{x_0'}}} & \ar[u]_{\RLBSch{\beta}{x_0'}} \ar[d]^{\RLBSch{\alpha}{x_0'}} \RLBSch{X}{x_0'} & \ar[l]_{i_{\RLBSch{X}{x_0'}}} \ar[u]_{\lBSch{\beta}{x_0'}} \ar[d]^{\lBSch{\alpha}{x_0'}} \lBSch{X}{x_0'} \ar@{=}[r]  & \ar[d]^{\alpha_{\lBSch{P}{x_0'}}} \ar[u]_{\beta_{\lBSch{P}{x_0'}}} \sch{X}_{\lBSch{P}{x_0'}} \\
			 \Lsch{L} \ar@{=}[r] & \LBSch{L}{x_0'} \ar[r]^{j_{\RLBSch{L}{x_0'}}} \ar[d] & \RLBSch{L}{x_0'} \ar[d] &\ar[l]_{i_{\RLBSch{L}{x_0'}}} \lBSch{L}{x_0'} \ar@{=}[r] \ar[d] & \lBSch{L}{x_0'}\\
	& \Spec{\Lq} \ar[r] & \Spec{\RLq} & \ar[l] \Spec{\lq} & \\
	}
\]
\caption{Parabolic induction and compact restriction}
\label{diagramme: ind}
\end{center}
\end{figure}

Consider Figure~\ref{diagramme: ind}, which consists entirely of cartesian squares.
Let $\fais{G}$ be a character sheaf of $\KKsch{L}$. Using the definition of parabolic induction \cite{lusztig:character-sheaves-I}*{4} and notation from \cite{cunningham-salmasian:sheaves}*{1.5.1}, we have
\begin{eqnarray*}
\cres{G}{x_0}\ \ind^\KKsch{G}_\KKsch{P} \fais{G} 
=
\NC{\RRSch{G}{x_0'}}\ {(\proj_1)}_!\ (\beta_\KKsch{P})_\# \ {(\alpha_\KKsch{P})}^*\ \fais{G}.
\end{eqnarray*}
Since the generic fibre of $\RLBSch{\alpha}{x_0'}$ is $\alpha_\LBsch{P}$ and $\RLBSch{\beta}{x_0'} = \beta_\LBsch{P}$, it follows that
	\begin{equation*}
 \NC{\RRSch{G}{x_0}}\ {(\proj_1)}_!\ (\beta_\KKsch{P})_\# \ {(\alpha_\KKsch{P})}^*\ \fais{G}
=
\NC{\RRSch{G}{x_0}}\ {(\proj_1)}_!\ (\LLBSch{\beta}{x_0'})_\#\ {\LLBSch{\alpha}{x_0'}}^*\ \fais{G}.
	\end{equation*}
The projection $\proj_1 : \RLBSch{Y}{x_0'} \to \RLBSch{G}{x_0'}$ is proper -- this is key! By proper base change, there is a  natural isomorphism
	\begin{equation*}
\NC{\RRSch{G}{x_0}}\ {(\proj_1)}_!\ (\LLBSch{\beta}{x_0'})_\#\ {(\LLBSch{\alpha}{x_0'})}^*\ \fais{G}
\iso  
{(\proj_1)}_!\ \NC{\RRLBSch{Y}{x_0'}}\ (\LLBSch{\beta}{x_0'})_\#\ {\LLBSch{\alpha}{x_0'}}^*\ \fais{G}.
	\end{equation*}
As explained in \cite{cunningham-salmasian:sheaves}*{\S1.4.5}, smooth base change provides a natural isomorphism
	\begin{equation*}
 {(\proj_1)}_!\ 
\NC{\RRLBSch{Y}{x_0'}}\ (\LLBSch{\beta}{x_0'})_\#\ {(\LLBSch{\alpha}{x_0'})}^*\ \fais{G}\iso {(\proj_1)}_!\ (\llBSch{\beta}{x_0'})_\#\ \NC{\RRLBSch{X}{x_0'}}\ {\LLBSch{\alpha}{x_0'}}^*\ \fais{G}.
	\end{equation*}
Recall that $\lBSch{\beta}{x_0'} =  \beta_{\lBSch{P}{x_0'}}$ and $\llBSch{\alpha}{x_0'} =  \alpha_\llBSch{P}{x_0'}$. Use smooth base change one more time:
	\begin{equation*}
 {(\proj_1)}_!\ 
(\llBSch{\beta}{x_0'})_\#\ \NC{\RRLBSch{X}{x_0'}}\ {\LLBSch{\alpha}{x_0'}}^*\ \fais{G}
 \iso  {(\proj_1)}_!\ (\beta_{\llBSch{P}{x_0'}})_\# \  {(\alpha_\llBSch{P}{x_0'})}^*\ 
 \NC{\RRSch{L}{x_0'}}\ \fais{G}.
	\end{equation*}
To finish, we need only recall the definition of induction (again) and compact restriction for the hyperspecial model $\RLBSch{L}{x_0'}$:
\begin{equation*}
{(\proj_1)}_!\ (\beta_{\llBSch{P}{x_0'}})_\# \  {(\alpha_\llBSch{P}{x_0'})}^*\ 
 \NC{\RRSch{L}{x_0'}}\ \fais{G}
= \ind^\llBSch{G}{x_0}_\llBSch{P}{x_0'} \ \cres{L}{x_0'} \fais{G}.
	\end{equation*}
\end{proof}

\section{Nearby cycles of character sheaves}

\begin{proposition}\label{proposition: BIRS}
Suppose $\Ksch{G}$ is a connected, reductive linear algebraic group over non-Archimedean local field $\Kq$ that satisfies hypotheses H.1 and H.2. Let $\RSch{G}{x_0}$ be a hyperspecial integral model for $G$. If $\fais{F}$ is a character sheaf of $\KKsch{G}$ then $\cres{G}{x_0}\fais{F}$ is a direct sum of character sheaves and thus a semisimple perverse sheaf of geometric origin.
\end{proposition}

\begin{proof}
Let $\fais{F}$ be an arbitrary character sheaf of $\KKsch{G}$. By  \cite{lusztig:character-sheaves-I}*{Theorem~4.4~(a)} (or \cite{mars-springer:character-sheaves}*{Corollary~9.3.5}) there is a parabolic subgroup $\KKsch{P}$ with Levi subgroup $\KKsch{L}$ and a cuspidal character sheaf $\fais{G}$ of $\KKsch{L}$ such that 
\begin{equation}\label{eqn: BIRS 1}
\ind_\KKsch{P}^\KKsch{G} \fais{G} = \mathop{\oplus}\limits_i \fais{F}_i \oplus \fais{F},
\end{equation}
where each $\fais{F}_i$ is a character sheaf of $\KKsch{G}$, and thus a simple perverse sheaf. Thus,
\begin{equation}\label{eqn: BIRS 2}
\cres{G}{x_0}\ \ind_\KKsch{P}^\KKsch{G} \fais{G} = \mathop{\oplus}\limits_i \cres{G}{x_0}\fais{F}_i \oplus \cres{G}{x_0}\fais{F}
\end{equation}
By hypothesis H.2, and using \cite{cunningham-salmasian:sheaves}*{Theorem~2} if necessary, we may assume $\KKsch{P} = \LBsch{P}\times_\Spec{\Lq} \Spec{\KKq}$ where $\Lq/\Kq$ is finite unramified, and that the image $x_0'$ of $x_0$ under $I(\Ksch{G},\Lq) \hookrightarrow I(\Ksch{G},\Lq)$ also lies in the image of $I(\Ksch{L},\Lq) \hookrightarrow I(\Ksch{G},\Lq)$. The hypotheses to Proposition~\ref{proposition: cres ind} are now met, so \begin{equation}\label{eqn: BIRS 3}
\cres{G}{x_0}\ \ind^\KKsch{G}_\KKsch{P} \fais{G}\iso  \ind^\kkSch{G}{x_0}_\kkSch{P}{x_0'} \ \cres{L}{x_0'} \fais{G}.
\end{equation}
By Proposition~\ref{proposition: semisimple}, $\cres{L}{x_0'} \fais{G}$ is a direct sum of character sheaves of $\kkSch{L}{x_0'} \fais{G}$. By \cite{lusztig:character-sheaves-I}*{Proposition~4.8~(b)} and \eqref{eqn: BIRS 3}, $\ind^\kkSch{G}{x_0}_\kkSch{\hskip-1pt P'}{x_0} \ \cres{L}{x_0'} \fais{G}$ is a direct sum of character sheaves. Thus, $\cres{G}{x_0}\ \ind_\KKsch{P}^\KKsch{G} \fais{G}$, is a direct sum of character sheaves. It now follows from \eqref{eqn: BIRS 2} that the simple consituents of $\cres{G}{x_0}\fais{F}$ are character sheaves.
\end{proof}

\section{Main result}

\begin{theorem}\label{theorem: 3bis}
Let $\Ksch{G}$ be a connected reductive linear algebraic group over $\Kq$ satisfying hypotheses H.1 and H.2. 
Let $\Lq/\Kq$ be a finite unramified extension.
Let $\LBsch{P}$ be a parabolic subgroup of $\Ksch{G}\times_\Spec{\Kq} \Spec{\Lq}$ with reductive quotient $\Ksch{L}\times_\Spec{\Kq} \Spec{\Lq}$ (so $\Ksch{L}$ is a `twisted Levi subgroup' of $\Ksch{G}$).
Let $x$ be an element in $I(\Ksch{G},\Kq)$. If the star of $x\in I(G,\K)$ contains a hyperspecial vertex then there is a finite set $\mathcal{S}\subset G(\Lq)$ such that
\[
\cres{G}{x}\ind^\KKsch{G}_\KKsch{P}\ \fais{G}
\iso \mathop{\oplus}\limits_{g\in \mathcal{S}} \ind^\rkkSch{G}{x}_{\rkk{G}{\,^g\hskip-1pt P}{x'}}\ \,^g\hskip-1pt\left(\cres{L}{x'g} \fais{G}\right),
\]
for every character sheaf $\fais{G}$ on $\KKsch{L}$. The finite set $\mathcal{S}\subset \Ksch{G}(\Lq)$, the parabolic subgroups $\rkk{G}{\,^g\hskip-1pt P}{x'}$ of $\rkkSch{G}{x'}$, the integral model $\RSch{L}{x'g}$ appearing in $\cres{L}{x'g}$, and the meaning of $\,^g\hskip-1pt(\cres{L}{x'g} \fais{G})$, are all given in the proof.
\end{theorem}

\begin{proof}
Let $x_0$ be a hyperspecial vertex in the star of $x$; then $x_0\leq x$. Using \cite{cunningham-salmasian:sheaves}*{Theorem~2}, we may assume $x_0\in I(\Ksch{L},\Kq)\hookrightarrow I(\Ksch{G},\Kq)$. By \cite{cunningham-salmasian:sheaves}*{Theorem~1}, there is a parabolic subgroup $\rkPSch{P}{x_0\leq x}$ of $\rkSch{G}{x_0} = \kSch{G}{x_0}$ with Levi component $\rkSch{G}{x}$ such that
\begin{equation}\label{equation: 3bis 1}
\cres{G}{x}\ind^\KKsch{G}_\KKsch{P}\ \fais{G}
\iso 
\res^\rkkSch{G}{x_0}_\rkkPSch{P}{x_0\leq x}\ \cres{G}{x_0}\ind^\KKsch{G}_\KKsch{P}\ \fais{G}.
\end{equation}
(The notation $\rkPSch{P}{x_0\leq x}$ is potentially confusing in the present context: the subgroup $\rkPSch{P}{x_0\leq x} \subseteq \rkSch{G}{x_0}$ is determined by $x_0$ and $x$ in $I(\Ksch{G},\Kq)$ and is unrelated to the subgroup $\LBsch{P} \subset \Lsch{G}$.)
Since $x_0$ is hyperspecial, it follows from Proposition~\ref{proposition: cres ind} that
\begin{equation}\label{equation: 3bis 2}
\cres{G}{x_0}\ind^\KKsch{G}_\KKsch{P}\ \fais{G}
\iso \ind^\kkSch{G}{x_0}_\kkSch{P}{x_0'} \ \cres{L}{x_0} \fais{G},
\end{equation}
where $x_0'$ is the image of $x_0$ under $I(\Ksch{G},\Kq) \hookrightarrow I(\Ksch{G},\Lq)$. 
(Note that we have replaced $\cres{L}{x_0'} \fais{G}$, as it appears in Proposition~\ref{proposition: cres ind}, with $\cres{L}{x_0} \fais{G}$ since $x_0\in I(\Ksch{L},\Kq)$.) Combining \eqref{equation: 3bis 1} and \eqref{equation: 3bis 2} gives
\begin{equation}\label{equation: 3bis 3}
\cres{G}{x}\ind^\KKsch{G}_\KKsch{P}\ \fais{G}
\iso \res^\rkkSch{G}{x_0}_\rkkPSch{P}{x_0\leq x}
 \ind^\kkSch{G}{x_0}_\kkSch{P}{x_0'} \ \cres{L}{x_0} \fais{G}.
 \end{equation}
By Proposition~\ref{proposition: BIRS} (which requires Hypothesis H.2) the perverse sheaf $\cres{L}{x_0} \fais{G}$ is a direct sum of character sheaves. Therefore, by the Mackey formula for character sheaves \cite{lusztig:character-sheaves-III}*{Proposition~15.2},
\begin{equation}\label{equation: 3bis 4}
\begin{aligned}
&\res^\kkSch{G}{x_0}_\rkkPSch{P}{x_0\leq x}
 \ind^\kkSch{G}{x_0}_\kkSch{P}{x_0'} \ \cres{L}{x_0} \fais{G}\\
& \iso
 \mathop{\oplus}\limits_{a\in \mathcal{S}(\rkkPSch{P}{x_0\leq x}, \lBSch{P}{x_0'})}
 \ind^\rkkSch{G}{x}_{\rkkSch{G}{x} \cap (\,^a\kkSch{P}{x_0'})}
\,^a\left(\res^\kkSch{L}{x_0}_{\kkSch{L}{x_0}\cap \sch{(\rkkPSch{P}{x_0\leq x})^a}}
 \cres{L}{x_0} \fais{G} \right)
 \end{aligned}
\end{equation}
where $\mathcal{S}(\rkkPSch{P}{x_0\leq x}, \lBSch{P}{x_0'})$ is a set of representatives $a\in \kkSch{G}{x_0}$ for double cosets 
\begin{equation}\label{equation: 3bis 5}
\rkkPSch{P}{x_0\leq x}(\kkq)\backslash \kkSch{G}{x_0}(\kkq) / \kkSch{P}{x_0'}(\kkq)
\end{equation}
such that $\kkSch{L}{x_0}$ and all $(\kkSch{P}{x_0'})^a \ceq a^{-1} \kkSch{P}{x_0'} a$ contain a common maximal torus of $\kkSch{G}{x_0}$ (not depending on $a$).

As explained in the proof of \cite{cunningham-salmasian:sheaves}*{Lemma~2}, $\rkkPSch{P}{x_0\leq x}$ is defined over $\kq$; in fact, $\rkkPSch{P}{x_0\leq x} = \rkPSch{P}{x_0\leq x}\times_\Spec{\kq} \Spec{\kkq}$ where $\rkPSch{P}{x_0\leq x}$ is defined in\cite{cunningham-salmasian:sheaves}*{Lemma~2}. Together with the fact that $\KKsch{P}$ is defined over $\Lq$ (by hypothesis), it follows  (as in \cite{digne-michel:lie-type}*{Lemma~5.6~(ii)}) that the double coset space above actually coincides with 
\begin{equation}\label{equation: 3bis 6}
\rkPSch{P}{x_0\leq x}(\lq)\backslash \kSch{G}{x_0}(\lq) / \lBSch{P}{x_0'}(\lq).
\end{equation}
 The surjective group homomorphism $\RSch{G}{x_0}(\RLq) \to \kSch{G}{x_0}(\lq)$ induces a bijection 
\begin{equation}\label{equation: 3bis 7}
\RSch{G}{x}(\RLq) \backslash \RSch{G}{x_0}(\RLq) / \RSch{P}{x_0'}(\RLq)
\to
\rkPSch{P}{x_0\leq x}(\lq)\backslash \kSch{G}{x_0}(\lq) / \lBSch{P}{x_0'}(\lq).
\end{equation}
We will use this bijection to replace the summation set appearing in \eqref{equation: 3bis 4} with a subset of $\Ksch{G}(\Lq)$ and to re-write the summands of  \eqref{equation: 3bis 4} in the form promised by Theorem~\ref{theorem: 3bis}. Let $x'$ be the image of $x$ under $I(\Ksch{G},\Kq)\hookrightarrow I(\Ksch{G},\Lq)$. 
For each $a\in \mathcal{S}(\rkPSch{P}{x_0\leq x}, \lBSch{P}{x_0})$ there is some $g\in \RSch{G}{x_0}(\RLq)$ such that: 
 \begin{itemize}
 \item[(i)]
 the image of $g$ under the surjective group homomorphism $\RSch{G}{x_0}(\RLq) \to \kSch{G}{x_0}(\lq)$ is $a$; 
 \item[(ii)]
  the reductive quotient $\rkSch{L}{x'g}$ of the special fibre of the schematic closure $\RSch{L}{x'g}$ of $\Ksch{L}$ in $\RSch{G}{x'g} \ceq \RSch{G}{g^{-1}x}$ is ${\kSch{L}{x_0}\cap (\rkPSch{P}{x_0\leq x})^a}$; and
  \item[(iii)]
  the image $\rk{G}{\,^g\hskip-1pt P}{x'}$ 
  of the special fibre of $\RSch{\,^g\hskip-1pt P}{x'}$ 
   under the map \[\nu_\RSch{G}{x'} : \kSch{G}{x'} \to \rkSch{G}{x'}\] is the Levi component of the parabolic subgroup ${\rkSch{G}{x'} \cap (\,^a\kSch{P}{x_0'})}$.
  \end{itemize}
Let $\mathcal{S}$ be a set of elements $g$ so chosen. 
  %
The double coset of $g\in \mathcal{S}$ is uniquely determined by the corresponding $a\in \mathcal{S}(\rkPSch{P}{x_0\leq x}, \lBSch{P}{x_0'})$. 

We now use \cite{cunningham-salmasian:sheaves}*{Theorem~1} to re-write
\begin{equation}\label{equation: 3bis 8}
\res^\kkSch{L}{x_0}_{\kkSch{L}{x_0}\cap \rkkPSch{P}{x_0\leq x}^a}
 \cres{L}{x_0} \fais{G} 
 =
 \cres{L}{x'g} \fais{G}.
 \end{equation}
 Because of the relationship between $g$ and $a$ articulated above, we also have  
\begin{equation}
\,^a (\cres{L}{x'g} \fais{G}) = \,^g (\cres{L}{x'g} \fais{G}) = (\kkSch{m(g^{-1})}{x})^* \cres{L}{x'g} \fais{G},
\end{equation}
where $\kkSch{m(g^{-1})}{x}$ is defined in \cite{cunningham-salmasian:sheaves}*{\S 2.3}, and
\begin{equation}
\ind^\rkkSch{G}{x}_{\rkkSch{G}{x} \cap (\,^a\kkSch{P}{x_0'})}
=
\ind^\rkkSch{G}{x}_{\rkk{G}{\,^g\hskip-1pt P}{x'}}.
\end{equation}
(Observe that 
Therefore, $\rkkSch{G}{x} = \rkkSch{G}{x'}$ because $\Lq/\Kq$ is unramified.)
Therefore,
\begin{equation}
\begin{aligned}
&\hskip-5pt
 \mathop{\oplus}\limits_{a\in \mathcal{S}(\rkPSch{P}{x_0\leq x}, \lBSch{P}{x_0'})}
 \ind^\rkkSch{G}{x}_{\rkkSch{G}{x} \cap (\,^a\kkSch{P}{x_0'})}
\,^a\left(\res^\kkSch{L}{x_0}_{\kkSch{L}{x_0}\cap(\rkkPSch{P}{x_0\leq x})^a}
 \cres{L}{x_0} \fais{G} \right)\\
&=
 \mathop{\oplus}\limits_{g\in\mathcal{S}} \ind^\rkkSch{G}{x}_{\rkk{G}{\,^g\hskip-1pt P}{x'}}\ \,^g\hskip-2pt \left(\cres{L}{x'g} \fais{G}\right),
 \end{aligned}
 \end{equation}
 thus completing the proof of Theorem~\ref{theorem: 3bis}.
\end{proof}

\begin{remark}\label{remark: H.0}
It was not necessary to impose Hypothesis H.0 on $\Ksch{G}$ at the beginning of the statement of Theorem~\ref{theorem: 3bis} because later we insisted that the star of $x$ contain a hyperspecial vertex, which has the effect of making Hypothesis H.0 true for $\Ksch{G}$. This is also the reason Hypothesis H.0 does not appear explicitly in Corollary~\ref{corollary: lusztig induction}.
\end{remark}

\begin{corollary}\label{corollary: lusztig induction}
Let $\Ksch{G}$ be a connected reductive linear algebraic group over $\Kq$ satisfying Hypotheses H.1 and H.2. 
Let $T\subset G$ be a maximal torus that splits over a tamely ramified extension $\Lq/\Kq$. Suppose $x\in I(\Ksch{G},\Kq)$. If the star of $x$ contains a hyperspecial vertex then there is a finite set $\mathcal{S} \subset G(\Lq)$ such that
\[
\cres{G}{x}\bar{K}^\mathcal{L}_e
\iso \mathop{\oplus}\limits_{g\in \mathcal{S}} \bar{K}_e^{\,^g\hskip-2pt(\NC{\RSch{T}{x'g}}\fais{L})}[\dim \rkSch{G}{x}- \dim\KKsch{G}/\KKsch{T}],
\]
for every Kummer local system $\fais{L}$ on $T_\KK$, where $\bar{K}^\fais{L}_e$ is the complex defined in \cite{lusztig:character-sheaves-III}*{\S 12.1} (and likewise, $\bar{K}_e^{\,^g\hskip-2pt(\NC{\RSch{T}{x'g}}\fais{L})}$).
\end{corollary}

\begin{proof}
Apply Theorem~\ref{theorem: 3bis} to the case when $\LBsch{P} = \LBsch{B}$ is a Borel subgroup of $\Ksch{G}\times_\Spec{\Kq} \Spec{\Lq}$ with Levi factor $\Ksch{T}\times_\Spec{\Kq} \Spec{\Lq}$. Use the fact that every character sheaf of $\KKsch{T}$ takes the form $\mathcal{L}[\dim T]$ for some Kummer local system $\fais{L}$ on $\KKsch{T}$ and $\bar{K}_e^\fais{L}[\dim G] = \ind_\KKsch{B}^\KKsch{G} \fais{L}[\dim T]$. Also use the fact that the smooth integral model $\RRSch{T}{x'g}$ (as defined in the proof of Theorem~\ref{theorem: 3bis}) is hyperspecial in the sense that $\kkSch{T}{x'g}$ is reductive, so $\cres{T}{x'g}{\fais{L}} = \NC{\RSch{T}{x'g}}\fais{L}$ 
 and $\dim\kkSch{T}{x'g} = \dim \KKsch{T}$ 
 for each $g\in \mathcal{S}$.
\end{proof}

\section{The full Mackey}

We believe Theorem~\ref{theorem: 3bis} is also true without the condition on $x\in I(\Ksch{G},\Kq)$ (that its star contains a hyperspecial vertex). Conjecture~\ref{conjecture: mackey}, below, is the topic of current work.

\begin{conjecture}[Mackey formula for compact restriction of character sheaves]  
\label{conjecture: mackey}
\ \newline Let $\Ksch{G}$ be a connected reductive linear algebraic group over $\Kq$ satisfying the Hypotheses H.1 and H.2. 
Let $\Lq/\Kq$ be a finite, tamely ramified extension.
Let $\LBsch{P}$ be a parabolic subgroup of $\Ksch{G}\times_\Spec{\Kq} \Spec{\Lq}$ with reductive quotient $\Ksch{L}\times_\Spec{\Kq} \Spec{\Lq}$ (so $\Ksch{L}$ is a `twisted Levi subgroup' of $\Ksch{G}$).
Let $x$ be an element in the Bruhat-Tits building $I(\Ksch{G},\Kq)$. 
There is a finite set $\mathcal{S}\subset G(\Lq)$ such that
\[
\cres{G}{x}\ind^\KKsch{G}_\KKsch{P}\ \fais{G}
\iso \mathop{\oplus}\limits_{g\in\mathcal{S}} \ind^\rkkSch{G}{x}_{\rkk{G}{\,^g\hskip-1pt P}{x'}}\ \,^g\hskip-2pt \left(\cres{L}{x'g} \fais{G}\right),
\]
for every character sheaf $\fais{G}$ on $\KKsch{L}$. 
The finite set $\mathcal{S}\subset \Ksch{G}(\Lq)$, the parabolic subgroups $\rkk{G}{\,^g\hskip-1pt P}{x'}$ of $\rkkSch{G}{x}$, and the integral model $\RSch{L}{x'g}$ appearing in $\cres{L}{x'g}$, are all as they appear in the proof of Theorem~\ref{theorem: 3bis}.
\end{conjecture}

In this paper we have proved Conjecture~\ref{conjecture: mackey} in the case that the star of $x$ contains a hyperspecial vertex.  Special linear groups and unitary groups, for example, have the property that for every $x\in I(\Ksch{G},\Kq)$ there is some hyperspecial vertex contained in the star of $x$, so Theorem~\ref{theorem: 3bis} can be used to determine $\cres{G}{x}\ind^\KKsch{G}_\KKsch{P} \fais{G}$ for every $x\in I(G,\Kq)$ in such cases. The smallest example of a group that does enjoy this property (that for every $x\in I(\Ksch{G},\Kq)$ there is some hyperspecial vertex contained in the star of $x$) is $\Ksch{G}=\Sp(4)$, precisely because the building for $\Sp(4,\Kq)$ contains non-hyperspecial vertices. In order to determine $\cres{G}{x}\ind^\KKsch{G}_\KKsch{P} \fais{G}$ in such cases, other techniques are required -- these are the topic of work in progress.



\appendix

\section{Toric $\Rq$-schemes}\label{appendix: toric schemes}

All schemes considered here will be separated schemes of finite type over $\Rq$.  
In particular, $\Gm$ denotes the group scheme $\Gm_{,\Rq} = \Spec{\Rq[t,t^{-1}]}$ and $\mathbb{A}^1$ denotes $\mathbb{A}^1_\Rq = \Spec{\Rq[t]}$.
Let $b: \Gm \to \mathbb{A}^1$ be the natural inclusion map.
Let $A$ be a finitely generated $\Rq$-algebra, and let $X = \Spec{A}$.

Specifying a group action $\mu: \Gm \times X \to X$ is equivalent to specifying an $\Rq$-module decomposition $A = \bigoplus_{n \in \ZZ} A_n$ that makes $A$ into a $\ZZ$-graded $\Rq$-algebra.
%
To see this, consider the coaction map $\mu^\sharp: A \to \Rq[t,t^{-1}] \otimes A$, and let $A_n = (\mu^\sharp)^{-1}(\Rq t^n \otimes A)$.  The claim follows from basic properties of $\mu^\sharp$.

Suppose now that $X$ is endowed with a $\Gm$-action.  Let $I \subset A$ be the ideal generated by $\Rq$-submodule $\bigoplus_{n \ne 0} A_n$, and set $X^\Gm = \Spec A/I$.  Let
\[
i: X^\Gm \to X
\]
denote the corresponding closed embedding.  $Z$ is the scheme of \emph{$\Gm$-fixed points}.  On the other hand, set $X \git \Gm = \Spec A_0$.  The corresponding map
\[
\pi: X \to X \git \Gm
\]
is called the \emph{invariant-theoretic quotient map}.
%
Let $X$ be a scheme with a $\Gm$-action $\mu: \Gm \times X \to X$.  This action is said to be \emph{contracting} if there is a map $\bar\mu: \mathbb{A}^1 \times X \to X$ such that the following diagram commutes:
\[
\xymatrix{
\Gm \times X \ar[rr]^{\mu} \ar[dr]_{b \times \id} && X \\
& \mathbb{A}^1 \times X \ar[ur]_{\bar\mu} }
\]
%
For an affine scheme $X = \Spec A$, an action $\mu: \Gm \times X \to X$ is contracting if and only if in the corresponding grading on $A$, we have $A_n = 0$ for $n < 0$.  When this holds, the map $\bar\mu: \mathbb{A}^1 \times X \to X$ is uniquely determined, and there is a canonical isomorphism $X^\Gm \cong X \git \Gm$.




\end{document}